\documentclass[11pt]{article}
\usepackage{pifont}

\usepackage{amsfonts}
\usepackage{latexsym}
\usepackage{amsmath}
\usepackage{amssymb}
\usepackage{color}

\newtheorem{theorem}{Theorem}[section]

\newtheorem{proposition}{Proposition}[section]

\newtheorem{corollary}{Corollary}[section]

\newenvironment{proof}[1][Proof]{\noindent \textbf{#1.} }{\ \ \  $\Box$}

\newtheorem{lemma}{Lemma}[section]
\newtheorem{definition}{Definition}[section]
\newtheorem{remark}{Remark}[section]

\title{Symmetrical Solutions
of Backward Stochastic Volterra Integral Equations and their
applications \thanks{This work is supported by National Natural
Science Foundation of China Grant 10771122, Natural Science
Foundation of Shandong Province of China Grant Y2006A08 and National
Basic Research Program of China (973 Program, No. 2007CB814900).}}

\date{May 28 2010}

\author{Tianxiao Wang and Yufeng Shi \thanks{E-mail: xiaotian2008001@gmail.com yfshi@sdu.edu.cn}\\ \small{School of
Mathematics, Shandong University, Jinan 250100, China}}

\begin{document}

\maketitle

\begin{abstract}
Backward stochastic Volterra integral equations (BSVIEs in short)
are studied. We introduce the notion of adapted symmetrical
solutions (S-solutions in short), which are different from the
M-solutions introduced by Yong \cite{Y2}. We also give some new
results for them. At last a class of dynamic coherent risk measures
were derived via certain BSVIEs.

 $\textit{Keywords:}$ Backward stochastic Volterra integral equations, Adapted
symmetrical solutions, Dynamic coherent risk measure
\end{abstract}



\section{Introduction}\label{sec:intro}

Let $(\Omega ,\mathcal{F},\mathbb{F},\mathbb{P})$ be a complete
filtered probability space on which a $d$-dimensional Brownian
motion $W(\cdot )$ is defined with $\mathbb{F}\equiv
\{{\mathcal{F}_t}\}_{t\geq 0}$ being its natural filtration
augmented by all the $\mathbb{P}$-null sets. In this paper, we
consider the following stochastic integral equations:
\begin{equation}
Y(t)=\Psi
(t)+\int_t^Tg(t,s,Y(s),Z(t,s),Z(s,t))ds-\int_t^TZ(t,s)dW(s), \quad
t\in [0,T].
\end{equation}
where $g:\Delta ^c\times R^m\times R^{m\times d}\times R^{m\times
d}\times \Omega \rightarrow R^m$ and $\Psi :[0,T]\times \Omega
\rightarrow R^m$ are some given measurable mappings with $\Delta
^c=\left\{ (t,s)\in [0,T]^2\mid t\leq s\right\} $. Such an equation
is referred as a backward stochastic Volterra integral equation
(BSVIE in short) (see
 \cite{Y1} and \cite{Y2}).

When $\Psi(\cdot)$, $g$, $Z(t,s)$ are independent of $t$, $g$ is
also independent of $Z(s,t)$, BSVIE (1) is reduced to a nonlinear
backward stochastic differential equation (BSDE in short)
\begin{equation*}
Y(t)=\xi+\int_t^Tg(s,Y(s),Z(s))ds-\int_t^TZ(s)dW(s),
\end{equation*}
which was introduced by Pardoux and Peng in \cite{PP1}, where the
existence and uniqueness of $\mathbb{F}$-adapted solutions are
established under uniform Lipschitz conditions on $g$. Due to their
important significance in many fields such as financial mathematics,
optimal control, stochastic games and partial differential equations
and so on, the theory of BSDEs has been extensively developed in the
past two decades. The reader is referred to \cite{EPQ}, \cite{HL},
\cite{MY}, \cite{PP2}, \cite{P1} and \cite{P2}.

On the other hand, stochastic Volterra integral equations were
firstly studied by Berger and Mizel in \cite{BM1} and \cite{BM2},
then they were investigated by Protter in \cite{P3} and Pardoux and
Protter in \cite{PP3}. Lin \cite{L} firstly introduced a kind of
nonlinear backward stochastic Volterra integral equations of the
form
\begin{equation*}
Y(t)=\xi+\int_t^Tg(t,s,Y(s),Z(t,s))ds-\int_t^TZ(t,s)dW(s).
\end{equation*}
But there is a gap in \cite{L}. At the same time Yong \cite{Y1} also
investigated a more general version of BSVIEs, as the type of (1),
and gave their applications to optimal control. In this paper, we
give a further discussion to BSVIE (1).

Based on the martingale presentation theorem, especially for
$Z(t,s),t\geq s$, Yong \cite{Y2} introduced the concept of
M-solutions for BSVIE (1). He gave some conditions that suffice
BSVIE (1) is uniquely solvable. We realize that there should be
other kinds of solutions for BSVIEs. In this paper, we introduce the
notion of symmetrical solutions (called S-solutions) in the way of
$Z(t,s)\equiv Z(s,t),t,s\in [0,T]$. It is worthy to point out that
S-solutions should be solved in a more general Hilbert space, which
is different from the one for M-solution, see more detailed accounts
in Section 2. We prove the existence and uniqueness of S-solutions
for BSVIEs. Some properties such as the continuity of $Y(t)$ are
obtained. We then study the relations between S-solutions and other
solutions. We give the notion of adapted solutions of (1) ($g$ is
independent of $Z(s,t)$) and obtain the existence and uniqueness by
virtue of the results of S-solution, which cover the ones in
\cite{L} and overcome its gap. Some relations between S-solutions
and M-solutions are studied. By two examples we show that the two
solutions usually are not equal, especially their values in
$\Delta=\left\{ (t,s)\in [0,T]^2\mid t>s\right\}.$ We also give a
criteria for S-solutions of BSVIEs. At last by giving a comparison
theorem for S-solutions of certain BSVIEs, we show a class of
dynamic coherent risk measures by means of S-solutions for certain
BSVIEs.

This paper is organized as follows: in the next section, we give
some preliminary results. In Section 3, we prove the existence and
uniqueness theorem of S-solutions of (1) and show some corollaries
and some other new results on S-solution. In Section 4 we give a
class of dynamic coherent risk measures by means of the S-solutions
of a kind of BSVIEs.


\section{Preliminary results}
\subsection{Notations and Definitions}

In this subsection we give some notations and definitions that are
needed in the following. For any $R,S\in[0,T],$ in the following we
denote $\Delta ^c[R,S]=\{(t,s)\in[R,S]^{2}; t\leq s\}$ and $\Delta
[R,S]=\{(t,s)\in[R,S]^{2}; t>s\}.$ Let $L_{\mathcal{F}_T}^2[0,T]$ be
the set of $\mathcal {B}([0,T])\otimes\mathcal{F}_T $-measurable
processes $X:[0,T]\times \Omega \rightarrow R^m$ satisfying
$$E\int_0^T|X(t)|^2dt<\infty .$$ We denote
\begin{equation*}
\mathbb{H}^2[R,S]=L_{\mathbb{F}}^2(\Omega ;C[R,S])\times L_{\mathbb{F}%
}^2(\Omega ;L^2[R,S]).
\end{equation*}
which is a Banach space under the norm:
\begin{equation*}
\Vert (y(\cdot ),z(\cdot ))\Vert _{\mathbb{H}^2[R,S]}=\left[
E\sup_{t\in [R,S]}|y(t)|^2+E\int_R^S|z(t)|^2dt\right] ^{\frac 12}.
\end{equation*}
Here $L_{\mathbb{F}}^2(\Omega ;C[R,S])$ is the set of all continuous adapted processes $%
X:[R,S]\times \Omega \rightarrow R^m$ satisfying
\begin{equation*}
E\left[ \sup_{t\in [R,S]}|X(t)|^2\right] <\infty.
\end{equation*}
$L_{\mathbb{F}}^2(\Omega ;L^2[R,S])$ is the set of all adapted processes $%
X:[R,S]\times \Omega \rightarrow R^{m\times d}$ satisfying
\begin{equation*}
E\int_R^S|X(t)|^2dt<\infty .
\end{equation*}
We denote
\begin{align*}
^{*}\mathcal{H}^2[R,S]=&L_{\mathbb{F}}^2[R,S]\times
L^2_{max;\mathbb{F}}(R,S;L^2[R,S]),
\\
\mathcal{H}^2[R,S]=&L_{\mathbb{F}}^2[R,S]\times L^2(R,S;L_{\mathbb{F}%
}^2[R,S]).
\end{align*}
Here $L^2_{max;\mathbb{F}}(R,S;L^2[R,S])$ is the set of all processes $%
z:[R,S]^2\times \Omega \rightarrow R^{m\times d}$ such that for almost all $%
t\in [R,S]$, $s\rightarrow z(t,s)$ is $\mathcal{F}_{t\vee
s}$-measurable satisfying
\begin{equation*}
E\int_R^S\int_R^S|z(t,s)|^2dsdt<\infty.
\end{equation*}
$L^2(R,S;L_{\mathbb{F}}^2[R,S])$ is the set of all processes $%
z:[R,S]^2\times \Omega \rightarrow R^{m\times d}$ such that for almost all $%
t\in [R,S]$, $z(t,\cdot )\in L_{\mathbb{F}}^2[R,S]$ satisfying
\begin{equation*}
E\int_R^S\int_R^S|z(t,s)|^2dsdt<\infty,
\end{equation*}
where $L_{\mathbb{F}}^2[R,S]=L_{\mathbb{F}}^2(\Omega ;L^2[R,S]).$

We also define the norm of the elements in $\mathcal{H}^2[R,S]$:
\begin{equation*}
\Vert (y(\cdot ),z(\cdot ,\cdot ))\Vert
_{\mathcal{H}^2[R,S]}=\left\{
E\int_R^S|y(t)|^2dt+E\int_R^S\int_R^S|z(t,s)|^2dsdt\right\} ^{1/2}.
\end{equation*}
As to the norm of the element in $^{*}\mathcal{H}^{2}[R,S],$
\begin{equation*}
\Vert (y(\cdot ),z(\cdot ,\cdot ))\Vert
_{^{*}\mathcal{H}^2[R,S]}=\left\{
E\int_R^S|y(t)|^2dt+E\int_R^S\int_R^S|z(t,s)|^2dsdt\right\} ^{1/2}.
\end{equation*}
From the definitions above, we know that $\mathcal{H}^2[R,S]$ is a
complete subspace of $^{*}\mathcal{H}^{2}[R,S]$ under the norm
above. Similarly we denote
\begin{equation*}
\mathcal{H}^2_1[R,S]=L_{\mathbb{F}}^2[R,S]\times L^2(R,S;L_{\mathbb{F}%
}^2[t,S]).
\end{equation*}
Here $L^2(R,S;L_{\mathbb{F}}^2[t,S])$ is the set of all processes $%
z:\Delta^c[R,S]\times \Omega \rightarrow R^{m\times d}$ such that for almost all $%
t\in [R,S]$, $z(t,\cdot )\in L_{\mathbb{F}}^2[R,S]$ satisfying
\begin{equation*}
E\int_R^S\int_t^S|z(t,s)|^2dsdt<\infty.
\end{equation*}

We also define the norm of the elements in $\mathcal{H}^2_1[R,S]$:
\begin{equation*}
\Vert (y(\cdot ),z(\cdot ,\cdot ))\Vert
_{\mathcal{H}^2_1[R,S]}=\left\{
E\int_R^S|y(t)|^2dt+E\int_R^S\int_t^S|z(t,s)|^2dsdt\right\} ^{1/2}.
\end{equation*}
 We now give the definition of M-solutions, introduced by Yong \cite{Y2}.
 \begin{definition}
Let $S\in [0,T]$. A pair of $(Y(\cdot ),Z(\cdot ,\cdot ))\in \mathcal{H}%
^2[S,T]$ is called an adapted M-solution of BSVIE (1) on $[S,T]$ if
(1) holds in the usual It\^o's sense for almost all $t\in [S,T]$
and, in addition, the following holds:
\begin{equation*}
Y(t)=E\left( \left. Y(t)\right| \mathcal{F}_S\right)
+\int_S^tZ(t,s)dW(s).
\end{equation*}
\end{definition}
In this paper, we introduce the concept of S-solutions as follows.
\begin{definition}
Let $S\in [0,T]$. A pair of $(Y(\cdot ),Z(\cdot ,\cdot ))\in  {^{*}\mathcal{H}}%
^2[S,T]$ is called an adapted symmetrical solution (also called
S-solution) of BSVIE (1) on $[S,T]$ if (1) holds in the usual
It\^o's sense for almost all $t\in [S,T]$ and, in addition, the
following holds:
\begin{equation*}
Z(t,s)=Z(s,t), \quad t,s\in [S,T].
\end{equation*}
\end{definition}
The following simple example shows the reason for considering the
S-solution in $^{*}\mathcal{H}^2[0,T]$, rather than
$\mathcal{H}^2[0,T].$ Let us consider the simple BSVIE
\begin{eqnarray*}
Y(t)=tW^2(T)-\int_t^T1ds-\int_t^TZ(t,s)dW(s).
\end{eqnarray*}
It is easy to show that $\forall (t,s)\in\Delta^c,$ $Y(t)=tW^2(t),$
$Z(t,s)=2tW(s)$ is the adapted solution of the above equation. By
the definition of S-solution, we have $\forall (t,s)\in\Delta,$
$Z(t,s)=Z(s,t)=sW(t),$ obviously we get
$(Y(\cdot),Z(\cdot,\cdot))\in^{*}\mathcal{H}^2[0,T],$ instead of
$\mathcal{H}^2[0,T].$

On the other hand, it is easy to see that $\forall (t,s)\in[0,T]^2$
and $(Y(\cdot),Z(\cdot,\cdot))\in\mathcal{H}^2[0,T]$, $Z(t,s)$ is
$\mathcal{F}_{0}$-measurable, i.e., almost surely deterministic
function if we consider the S-solution of BSVIE (1) in
$\mathcal{H}^2[0,T].$ In fact, in this case, for any
$(t,s)\in[0,T]^2,$ $Z(t,s)$ ($Z(s,t)$ respectively) is
$\mathcal{F}_{s}$-measurable ($\mathcal{F}_{t}$-measurable
respectively), then by $Z(t,s)=Z(s,t)$ we obtain that $Z$ should be
$\mathcal{F}_{0}$-measurable.

Now we cite some definitions introduced in \cite{Y3}.
\begin{definition}
A mapping $\rho :L_{\mathcal{F}_T}^2[0,T]\rightarrow
L_{\mathbb{F}}^2[0,T]$ is called a dynamic risk measure if the
following hold:

1) (Past independence) For any $\Psi (\cdot ),$ $\overline{\Psi }(\cdot )\in L_{\mathcal{F}%
_T}^2[0,T],$ if $\Psi (s)=\overline{\Psi }(s),$ a.s. $\omega \in \Omega ,$ $%
s\in [t,T],$ for some $t\in [0,T),$ then $\rho (t;\Psi (\cdot ))=\rho (t;%
\overline{\Psi }(\cdot )),$ a.s. $\omega \in \Omega .$

2) (Monotonicity) For any $\Psi (\cdot ),$ $\overline{\Psi }(\cdot )\in L_{\mathcal{F}%
_T}^2[0,T],$ if $\Psi (s)\leq \overline{\Psi }(s),$ a.s. $\omega \in
\Omega , $ $s\in [t,T],$ for some $t\in [0,T),$ then $\rho (s;\Psi
(\cdot ))\geq \rho (s;\overline{\Psi }(\cdot )),$ a.s. $\omega \in
\Omega, \text{ } s\in[t,T].$
\end{definition}
\begin{definition}
A dynamic risk measure $\rho :L_{\mathcal{F}_T}^2[0,T]\rightarrow L_{\mathbb{%
F}}^2[0,T]$ is called a coherent risk measure if the following hold:
1) There exists a deterministic integrable function $r(\cdot )$ such
that for
any $\Psi (\cdot )\in L_{\mathcal{F}_T}^2[0,T],$%
\begin{equation*}
\rho (t;\Psi (\cdot )+c)=\rho (t;\Psi (\cdot
))-ce^{\int_t^Tr(s)ds},\text{ a.s. }\omega \in \Omega ,\text{ }t\in
[0,T].
\end{equation*}
2) For $\Psi (\cdot )\in L_{\mathcal{F}_T}^2[0,T]$ and $\lambda >0,$
$\rho (t;\lambda \Psi (\cdot ))=\lambda \rho (t;\Psi (\cdot ))$ a.s.
$\omega \in \Omega ,$ $t\in [0,T].$
3) For any $\Psi (\cdot ),$ $\overline{\Psi }(\cdot )\in L_{\mathcal{F}%
_T}^2[0,T],$
\begin{equation*}
\rho (t;\Psi (\cdot )+\overline{\Psi }(\cdot ))\leq \rho (t;\Psi
(\cdot
))+\rho (t;\overline{\Psi }(\cdot )),\text{ a.s. }\omega \in \Omega ,\text{ }%
t\in [0,T].
\end{equation*}
\end{definition}

\subsection{Some lemmas for S-solutions}

First we give some lemmas for S-solutions. For any $R,S\in [0,T],$
let us consider the following stochastic integral equation
\begin{equation}
\lambda (t,r)=\Psi (t)+\int_r^Th(t,s,\mu (t,s))ds-\int_r^T\mu
(t,s)dW(s),\quad r\in [S,T],\quad t\in [R,T].
\end{equation}
where $h:[R,T]\times [S,T]\times R^{m\times d}\times \Omega
\rightarrow R^m$ is given. The unknown processes are $(\lambda
(\cdot ,\cdot ),\mu (\cdot ,\cdot ))$,
for which $%
(\lambda (t,\cdot ),\mu (t,\cdot ))$ are $\mathbb{F}$-adapted for
all $t\in [R,T]$. We can regard (2) as a family of BSDEs on $[S,T]$,
parameterized by $t\in [R,T]$. Next we introduce the following
assumption of $h$ in (2).

(H1) Let $R,S\in [0,T]$ and $h:[R,T]\times [S,T]\times R^{m\times
d}\times \Omega \rightarrow R^m$ be $\mathcal{B}([R,T]\times
[S,T]\times R^{m\times
d})\otimes \mathcal{F}_T$-measurable such that $s\mapsto h(t,s,z)$ is $%
\mathbb{F}$-progressively measurable for all $(t,z)\in [R,T]\times
R^{m\times d}$ and
\begin{equation}
E\int_R^T\left( \int_S^T|h(t,s,0)|ds\right) ^2dt<\infty .
\end{equation}
Moreover, the following holds:
\begin{equation}
|h(t,s,z_1)-h(t,s,z_2)|\leq L(t,s)|z_1-z_2|,(t,s)\in [R,T]\times
[S,T],z_1,z_2\in R^{m\times d},
\end{equation}
where $L:[R,T]\times [S,T]\rightarrow [0,\infty )$ is a
deterministic function such that for some $\varepsilon >0$,
$$\sup\limits_{t\in [R,T]}\int_S^TL(t,s)^{2+\epsilon }ds<\infty .$$
\begin{proposition}
Let (H1) hold, then for any $\Psi (\cdot )\in
L_{\mathcal{F}_T}^2[R,T]$, (2) admits a unique adapted solution
$(\lambda (t,\cdot ),\mu (t,\cdot ))\in \mathbb{H}^2[S,T]$ for
almost all $t\in [R,T],$ and the following estimate holds: $t\in
[R,T]$
\begin{eqnarray}
\left\| (\lambda (t,\cdot ),\mu (t,\cdot ))\right\| _{\mathbb{H}%
^2[S,T]}^2\equiv E\left\{ \sup_{r\in [S,T]}|\lambda
(t,r)|^2+\int_S^T|\mu
(t,s)|^2ds\right\}  \nonumber \\
\ \leq CE\left\{ |\Psi (t)|^2+\left( \int_S^T|h(t,s,0)|ds\right)
^2\right\}.
\end{eqnarray}
\end{proposition}
\begin{proof}The proof of Proposition 1 can be found in \cite{Y2}.
\end{proof}

Now we look at one special case of (2). Let $R=S$ and define
\begin{equation*}
\left\{
\begin{array}{lc}
Y(t)=\lambda (t,t), & t\in [S,T], \\
Z(t,s)=\mu (t,s), & (t,s)\in \Delta ^c[S,T].%
\end{array}
\right.
\end{equation*}
Then the above (2) reads:
\begin{equation}
Y(t)=\Psi (t)+\int_t^Th(t,s,Z(t,s))ds-\int_t^TZ(t,s)dW(s),\quad t\in
[S,T].
\end{equation}

Here we define $Z(t,s)$ for $(t,s)\in \Delta[S,T]$ by the following
relation $Z(t,s)=Z(s,t)$, which is different from the
 way of defining M-solution, and that's why we call it S-solution
  of (6). So
we have the following lemma.

\begin{lemma}
Let (H1) hold, then for any $S\in [0,T],$ $\Psi (\cdot )\in L_{\mathcal{F}%
_T}^2[S,T]$, (6) admits a unique adapted S-solution $(Y(\cdot
),Z(\cdot ,\cdot ))\in ^{*}\mathcal{H}^2[S,T]$, and the following
estimate holds: $t\in [S,T],$
\begin{equation}
E\left\{ |Y(t)|^2+\int_t^T|Z(t,s)|^2ds\right\} \leq CE\left\{ |\Psi
(t)|^2+\left( \int_t^T|h(t,s,0)|ds\right) ^2\right\}.
\end{equation}
Hereafter $C$ is a generic positive constant which may be different from line to line. If $\overline{h}$ also satisfies (H1), $\overline{\Psi }(\cdot )\in L_{%
\mathcal{F}_{T}}^{2}[S,T],$ and $(\overline{Y}(\cdot
),\overline{Z}(\cdot ,\cdot ))\in ^{*}\mathcal{H}^{2}[S,T]$ is the
unique adapted S-solution of BSVIE
(6) with $(h,\Psi )$ replaced by $(\overline{h},\overline{\Psi }),$ then $%
\forall t\in \lbrack S,T],$
\begin{eqnarray}
&&\ E\left\{ |Y(t)-\overline{Y}(t)|^{2}+\int_{t}^{T}|Z(t,s)-\overline{Z}%
(t,s)|^{2}ds\right\}  \nonumber \\
\  &\leq &CE\left[ |\Psi (t)-\overline{\Psi }(t)|^{2}+\left(
\int_{t}^{T}|h(t,s,Z(t,s))-\overline{h}(t,s,Z(t,s))|ds\right)
^{2}\right] .
\end{eqnarray}%
Furthermore, for any $t,\overline{t}\in \lbrack S,T]$,
\begin{eqnarray}
&& E\left\{ |Y(t)-Y(\overline{t})|^{2}+\int_{t\vee \overline{t}}^{T}|Z(t,s)-Z(%
\overline{t},s)|^{2}ds\right\} \nonumber \\
&\leq & CE\left\{ |\Psi (t)-\Psi (\overline{t})|^{2}+\left(
\int_{t\wedge \overline{t}}^{t\vee \overline{t}}|h(t\wedge
\overline{t},s,Z(t\wedge
\overline{t},s))|ds\right) ^{2}\right. \nonumber \\
&&+ \left. \int_{t\wedge \overline{t}}^{t\vee
\overline{t}}|Z(t\wedge
\overline{t},s)|^{2}ds+\left( \int_{t\vee \overline{t}}^{T}|h(t,s,Z(t,s))-h(%
\overline{t},s,Z(t,s))|ds\right) ^{2}\right\} .
\end{eqnarray}
\end{lemma}
\begin{proof} From Proposition 1 the existence and uniqueness of S-solution in $[S,T]$
is clear. As to the other estimates, the proof is the same as the
one in \cite{Y2}.
\end{proof}

Let's give another special case. Let $r=S\in [R,T]$ be fixed. Define
\begin{equation*}
\Psi ^S(t)=\lambda (t,S),\quad Z(t,s)=\mu (t,s),\text{ \quad }t\in [R,S],%
\text{ \quad }s\in [S,T].
\end{equation*}
Then (2) becomes:
\begin{equation}
\Psi ^S(t)=\Psi
(t)+\int_S^Th(t,s,Z(t,s))ds-\int_S^TZ(t,s)dW(s),\quad t\in [R,S],
\end{equation}
and we have the following result.

\begin{lemma}
Let (H1) hold, then for any $\Psi (\cdot )\in
L_{\mathcal{F}_T}^2[R,S]$, $(10)$ admits a unique adapted solution
$(\Psi ^S(\cdot ),Z(\cdot ,\cdot ))\in
L_{\mathcal{F}_S}^2[R,S]\times L^2(R,S;L_{\mathbb{F}}^2[S,T])$, and
the following estimate holds: $ t\in [R,S],$
\begin{equation}
E\left\{ |\Psi ^S(t)|^2+\int_S^T|Z(t,s)|^2ds\right\} \leq CE\left\{
|\Psi (t)|^2+\left( \int_S^T|h(t,s,0)|ds\right) ^2\right\}.
\end{equation}
\end{lemma}

\begin{proof}From Proposition 1 the result is obvious. \end{proof}


\section{Well-posedness of S-solutions for BSVIEs}


\subsection{The existence and uniqueness of S-solutions}

In this subsection we will give the existence and uniqueness of
S-solutions. For it we need the following standing assumption.

(H2) Let $g:\Delta ^c\times R^m\times R^{m\times d}\times R^{m\times
d}\times \Omega \rightarrow R^m$ be $\mathcal{B}(\Delta ^c\times
R^m\times R^{m\times d}\times R^{m\times d})\otimes
\mathcal{F}_T$-measurable such that $s\rightarrow g(t,s,y,z,\zeta )$
is $\mathbb{F}$-progressively measurable for all $(t,y,z,\zeta )\in
[0,T]\times R^m\times R^{m\times d}\times R^{m\times d}$ and
\begin{equation}
E\int_0^T\left( \int_t^T|g_0(t,s)|ds\right)^2dt<\infty ,
\end{equation}
where we denote $g_0(t,s)\equiv g(t,s,0,0,0).$ Moreover, it holds
\begin{eqnarray*}
|g(t,s,y,z,\zeta)-g(t,s,\overline{y},\overline{z},\overline{\zeta
})| &\leq
&L(t,s)\left( |y-\overline{y}|+|z-\overline{z}|+|\zeta -\overline{\zeta }%
|\right) ,  \text{ a.s.} \\
\forall (t,s) &\in &\Delta ^c,\quad y,\overline{y}\in R^m,\quad z,\overline{z%
},\zeta ,\overline{\zeta }\in R^{m\times d},
\end{eqnarray*}
where $L:\Delta ^c\rightarrow R$ is a deterministic function such
that the following holds: for some $\epsilon >0,$
\begin{equation*}
\sup_{t\in [0,T]}\int_t^TL(t,s)^{2+\epsilon }ds<\infty .
\end{equation*}
So we have:
\begin{theorem}
Let (H2) hold, then for any $\Psi(\cdot)\in
L_{\mathcal{F}_T}^2[0,T]$, (1) admits a unique adapted S-solution on
$[0,T]$. Moreover, the following estimate holds: $\forall S\in
[0,T],$
\begin{eqnarray}
\left\| (Y(\cdot ),Z(\cdot,\cdot ))\right\|
_{^{*}\mathcal{H}^2[S,T]}^2\equiv
E\left\{\int^T_S|Y(t)|^2dt+\int_S^T\int_S^T|Z (t,s)|^2dsdt\right\}  \nonumber \\
\ \leq CE\left\{ \int_S^T|\Psi (t)|^2dt+\int_S^T\left(
\int_t^T|g_{0}(t,s)|ds\right) ^2dt\right\} .
\end{eqnarray}
\end{theorem}

\begin{proof}
We split the proof into two steps.

\textbf{Step1} Here we consider the existence and uniqueness of the
adapted S-solution of (1) on $[S,T]$ for some $S\in [0,T].$ For all
$S\in [0,T],$ let $\mathcal {S}^2[S,T]$ be the space of all
$(y(\cdot
),z(\cdot ,\cdot ))\in {^{*}\mathcal {H}^2[S,T]}$ such that $$z(t,s)=z(s,t), a.s., \quad  t, %
s\in [S,T],  a.e..$$  Clearly, $\mathcal {S}^2[S,T]$ is a nontrivial
closed subspace of $^{*}\mathcal {H}^2[S,T]$. In fact, we assume
that there is a
series of elements $(y_n(\cdot ),z_n(\cdot ,\cdot ))$ in $\mathcal {S}%
^2[S,T],$ and the limit is $(y(\cdot ),z(\cdot ,\cdot )),$ which
belongs to $^{*}\mathcal {H}^2[S,T].$ We easily know that the limit
also belongs to $\mathcal {S}^2[S,T]$. Actually, we have the
following:
\begin{eqnarray}
&&E\int_S^T\int_S^T|z(t,s)-z(s,t)|^2dsdt \nonumber \\
&\leq
&E\int_S^T\int_S^T|z(t,s)-z_n(t,s)|^2dsdt\nonumber \\
&&+ E\int_S^T%
\int_S^T|z_n(s,t)-z(s,t)|^2dsdt.
\end{eqnarray}
As $n\rightarrow \infty,$ the limit of the right hand of (14) is
zero, so we have
\[
z(t,s)=z(s,t), \text{a.s.}, \quad t,s\in [S,T],\text{a.e.}.
\]
Note that here the space $\mathcal{S}^2[0,T]$ is isomorphic to
$\mathcal{H}_1^2[0,T]$ defined previously, i.e., there exists a
bijection between them. For any $(y(\cdot ),z(\cdot ,\cdot ))\in
\mathcal {S}^2[S,T]$, we have
\begin{eqnarray*}
&&E\left\{ \int_S^T|y(t)|^2dt+\int_S^T\int_t^T|z(t,s)|^2dsdt\right\} \\
&\leq &E\left\{
\int_S^T|y(t)|^2dt+\int_S^T\int_S^T|z(t,s)|^2dsdt\right\}\\
&=&E\left\{
\int_S^T|y(t)|^2dt+\int_S^T\int_S^t|z(t,s)|^2dsdt+\int_S^T%
\int_t^T|z(t,s)|^2dsdt\right\} \\
&=&E\left\{
\int_S^T|y(t)|^2dt+\int_S^T\int_s^T|z(t,s)|^2dtds+\int_S^T%
\int_t^T|z(t,s)|^2dsdt\right\} \\
 &=&E\left\{
\int_S^T|y(t)|^2dt+\int_S^T\int_s^T|z(s,t)|^2dtds+\int_S^T%
\int_t^T|z(t,s)|^2dsdt\right\} \\
&=&E\left\{
\int_S^T|y(t)|^2dt+\int_S^T\int_t^T|z(t,s)|^2dsdt+\int_S^T%
\int_t^T|z(t,s)|^2dsdt\right\} \\
&\leq &2E\left\{
\int_S^T|y(t)|^2dt+\int_S^T\int_t^T|z(t,s)|^2dsdt\right\}.
\end{eqnarray*}
Hence, we can take a new norm for the elements of ${\mathcal
{S}}^2[S,T]$ as follows:
\[
\Vert (y(\cdot ),z(\cdot ,\cdot ))\Vert _{\mathcal {S}^2[S,T]}\equiv
E\left\{ \int_S^T|y(t)|^2dt+\int_S^T\int_t^T|z(t,s)|^2dsdt\right\}
^{\frac 12}.
\]
Now we consider the following equation: $ t\in [S,T],$
\begin{equation}
Y(t)=\Psi
(t)+\int_t^Tg(t,s,y(s),Z(t,s),z(s,t))ds-\int_t^TZ(t,s)dW(s),
\end{equation}
for any $\Psi (\cdot )\in L_{\mathcal {F}_T}^2[S,T]$ and $(y(\cdot
),z(\cdot ,\cdot ))\in \mathcal {S}^2[S,T]$. By Lemma 2.5, (15)
admits a unique adapted S-solution $(Y(\cdot ),Z(\cdot ,\cdot ))\in
\mathcal {S}^2[S,T]$ and we can define a mapping $\Theta :{\mathcal
{S}}^2[S,T]\rightarrow \mathcal {S}^2[S,T]$ by
\[
\Theta (y(\cdot ),z(\cdot ,\cdot ))=(Y(\cdot ),Z(\cdot ,\cdot
)),\quad \forall (y(\cdot ),z(\cdot ,\cdot ))\in \mathcal
{S}^2[S,T].
\]
Next we will prove $\Theta $ defined above is contracted when
$T-S>0$ is small enough. Let $(\overline{y}(\cdot
),\overline{z}(\cdot ,\cdot ))\in \mathcal {S}^2[S,T]$ and $\Theta
(\overline{y}(\cdot ),\overline{z}(\cdot ,\cdot
))=(\overline{Y}(\cdot ),\overline{Z}(\cdot ,\cdot )).$ From (8) we
know that
\begin{eqnarray}
&&E|Y(t)-\overline{Y}(t)|^2+E\int_t^T|Z(t,s)-\overline{Z}(t,s)|^2ds
\nonumber \\
&\leq &CE\left( \int_t^T|g(t,s,y(s),Z(t,s),z(s,t))-g(t,s,\overline{y}%
(s),Z(t,s),\overline{z}(s,t))|ds\right) ^2  \nonumber \\
&\leq &CE\left\{ \int_t^TL(t,s)(|y(s)-\overline{y}(s)|+|z(t,s)-\overline{z}%
(t,s)|)ds\right\} ^2  \nonumber \\
&\leq &C(T-t)^{\frac \epsilon
{2+\epsilon }}\sup_{t\in [0,T]}\left( \int_t^TL(t,s)^{2+\epsilon
}ds\right) ^{\frac 2{2+\epsilon }}E\left\{
\int_t^T|y(s)-\overline{y}(s)|^2dt \right\}\nonumber \\
&&+C(T-t)^{\frac \epsilon {2+\epsilon }}\sup_{t\in [0,T]}\left(
\int_t^TL(t,s)^{2+\epsilon }ds\right) ^{\frac 2{2+\epsilon
}}E\left\{ \int_t^T|z(t,s)-\overline{z}(t,s)|^2ds  \right\}.
\nonumber
\\
\end{eqnarray}
The second inequality in (16) holds because of $z(t,s)\equiv
z(s,t).$ Consequently,
\begin{align*}
&\left\| (Y(\cdot ),Z(\cdot ,\cdot ))-(\overline{Y}(\cdot ),\overline{Z}%
(\cdot ,\cdot ))\right\| _{\mathcal {S}^2[S,T]}^2 \\
\equiv &E\left\{ \int_S^T|Y(t)-\overline{Y}(t)|^2dt+\int_S^T\int_t^T|Z(t,s)-%
\overline{Z}(t,s)|^2dsdt\right\} \\
\leq &C(T-S)^{\frac {\epsilon} {2+\epsilon }+1}\sup_{t\in
[0,T]}\left( \int_t^TL(t,s)^{2+\epsilon }ds\right) ^{\frac
2{2+\epsilon }}E\left\{
\int_S^T|y(t)-\overline{y}(t)|^2dt \right\}\nonumber \\
&+C(T-S)^{\frac \epsilon {2+\epsilon }}\sup_{t\in [0,T]}\left(
\int_t^TL(t,s)^{2+\epsilon }ds\right) ^{\frac 2{2+\epsilon
}}E\left\{\int_S^T\int_t^T|z(t,s)-\overline{z}(t,s)|^2dsdt \right\}.
\nonumber
\end{align*}
Then we can choose $\eta$ so that $C' \max\{\eta ^{\frac \epsilon
{2+\epsilon }}, \eta ^{\frac \epsilon {2+\epsilon }+1}\} =\frac 12,$
where $$C'=C\sup_{t\in [0,T]}\left( \int_t^TL(t,s)^{2+\epsilon
}ds\right) ^{\frac 2{2+\epsilon }}.$$ Hence (15) admits a unique
fixed point $(Y(\cdot ),Z(\cdot ,\cdot ))\in \mathcal {S}^2[S,T]$
which is the unique adapted S-solution of (1) in $[S,T]$ if $T-S\leq
\eta$.

{\bf Step 2:} Now we will prove the existence and uniqueness of the
S-solution of (1) for $(t,s)\in [R,S]\times [R,S],$ for some $%
R\in [0,S].$ First we consider the following equation: $ t\in
[R,S],$
\begin{equation}
Y(t)=\Psi^S(t)+\int_t^Sg(t,s,Y(s),Z(t,s),Z(s,t))ds-\int_t^SZ(t,s)dW(s),
\end{equation}
where
\begin{equation}
\Psi^S(t)=\Psi(t)+\int_S^Tg(t,s,Y(s),Z(t,s),Z(s,t))ds-\int_S^TZ(t,s)dW(s).
\end{equation}
If we prove that $\Psi ^S(t)$ is $\mathcal {F}_S$-measurable and
$\Psi ^S(t)\in L_{\mathcal {F}_S}^2[R,S]$, then we can use the same
argument as Step 1 to show
(1) is solvable on $[R,S]$ when $\eta\geq S-R>0$ is small enough. Let $%
Z(t,s)\equiv Z(s,t) $ in (18), and we denote $$%
h(t,s,Y(s),Z(t,s))\equiv g(t,s,Y(s),Z(t,s),Z(s,t)).$$ From Step 1 we
have known that
$\{Y(s); s\in [S,T]\}$ is solved, then by Lemma 2.6, (18) admits a unique adapted solution $%
(\Psi ^S(\cdot ),Z(\cdot ,\cdot ))\in L_{\mathcal {F}_S}^2[R,S]\times L^2(R,S;\break L_{%
{\Bbb F}}^2[S,T]),$ and
\begin{eqnarray}
&&E\left[ \int_R^S|\Psi
^S(t)|^2dt+\int_R^S\int_S^T|Z(t,s)|^2dsdt\right] \nonumber \\
&\leq &CE\int_R^S|\Psi (t)|^2dt+CE\int_R^S\left(
\int_S^T|h(t,s,Y(s),0)|ds\right) ^2dt  \nonumber\\
&=&CE\int_R^S|\Psi (t)|^2dt+CE\int_R^S\left(
\int_S^T|g(t,s,Y(s),0,0)|ds\right) ^2dt  \nonumber \\
&\leq &CE\int_R^S|\Psi (t)|^2dt+CE\int_R^S\left(
\int_S^T|g_0(t,s)|ds\right)
^2dt  \nonumber \\
&&+CE\int_R^S\left( \int_S^TL(t,s)|Y(s)|ds\right) ^2dt.
\end{eqnarray}
Here $C$ is a constant depending on $\sup\limits_{t\in
[0,T]}\int_t^TL(t,s)^{2+\epsilon }ds$ and $T.$ Then we have
$E\int_R^S|\Psi ^S(t)|^2dt <\infty.$ Thus we can repeat the argument
in Step 2 to finish the proof of the existence and uniqueness of
adapted S-solution of (1) on $[0,T]$.

Next we prove the estimate in the theorem. First we can choose
$T_1\in [0,T]$ so that it satisfies $\max \left\{ 8\theta ^{\frac
\epsilon {2+\epsilon }}A,4\theta ^{\frac \epsilon {2+\epsilon
}+1}A\right\} =\frac 12$, where
\[
\theta =T-T_1,\text{  }  \text{  } A=\sup\limits_{t\in [0,T]}\left(
\int_t^TL^{2+\epsilon }(t,s)ds\right) ^{\frac 2{2+\epsilon }},
\]
then from BSVIE (1) we have, $\forall u\in [T_1,T],$
\begin{eqnarray*}
&&E\int_u^T|Y(t)|^2dt+E\int_u^T\int_t^T|Z(t,s)|^2dsdt \\
&\leq& 2E\int_u^T|\Psi (t)|^2dt+4E\int_u^T\left(
\int_t^T|g_0(t,s)|ds\right) ^2dt \nonumber \\
&&+ 4E\int_u^T\left(
\int_t^T|g(t,s,Y(s),Z(t,s),Z(s,t))-g_0(t,s)|ds\right) ^2dt \nonumber
\\
 &\leq& 2E\int_u^T|\Psi (t)|^2dt+4E\int_u^T\left(
\int_t^T|g_0(t,s)|ds\right) ^2dt \nonumber \\
&&+ 4(T-T_1)^{\frac \epsilon {2+\epsilon }+1}\sup\limits_{t\in
[0,T]}\left( \int_t^TL^{2+\epsilon }(t,s)ds\right) ^{\frac
2{2+\epsilon }}E\int_u^T|Y(s)|^2ds  \nonumber \\
&&+ 8(T-T_1)^{\frac \epsilon {2+\epsilon }}\sup\limits_{t\in
[0,T]}\left( \int_t^TL^{2+\epsilon }(t,s)ds\right) ^{\frac
2{2+\epsilon }}E\int_u^T\int_t^T|Z(t,s)|^2dsdt, \nonumber
\end{eqnarray*}
thus by the way of choosing $T_{1}$ we know that
\begin{eqnarray}
&&\ E\int_u^T|Y(t)|^2dt+E\int_u^T\int_t^T|Z(t,s)|^2dsdt  \nonumber \\
&\leq &4E\int_u^T|\Psi (t)|^2dt+8E\int_u^T\left(
\int_t^T|g_0(t,s)|ds\right) ^2dt,
\end{eqnarray}
furthermore we have, $\forall u\in [T_1,T],$
\begin{eqnarray}
&&\ E\int_u^T|Y(t)|^2dt+E\int_u^T\int_u^T|Z(t,s)|^2dsdt   \nonumber \\
&=&E\int_u^T|Y(t)|^2dt+2E\int_u^T\int_t^T|Z(t,s)|^2dsdt  \nonumber \\
&\leq &CE\int_u^T|\Psi (t)|^2dt+CE\int_u^T\left(
\int_t^T|g_0(t,s)|ds\right) ^2dt.
\end{eqnarray}
Similarly we can choose $T_2\in [0,T_1]$ satisfying $T_1-T_2=\theta
=T-T_1,$ so that $u\in [T_2,T_1].$
\begin{eqnarray}
&&E\left\{
\int_u^{T_1}|Y(t)|^2dt+\int_u^{T_1}\int_u^{T_1}|Z(t,s)|^2dsdt\right\}\nonumber
\\
&\leq &CE\left\{ \int_u^{T_1}|\Psi ^{T_1}(t)|^2dt+\int_u^{T_1}\left(
\int_t^{T_1}|g_0(t,s)|ds\right) ^2dt\right\} ,
\end{eqnarray}
where
\begin{eqnarray}
\Psi ^{T_1}(t)=\Psi
(t)+\int_{T_1}^Tg(t,s,Y(s),Z(t,s),Z(s,t))ds-\int_{T_1}^TZ(t,s)dW(s).
\end{eqnarray}
By the definition of S-solution we have $Z(t,s)\equiv Z(s,t)$ in
(23),
then from the proof in Step 2 above we have, $\forall u\in [T_2,T_1],$%
\begin{eqnarray}
&&\ E\left[ \int_u^{T_1}|\Psi
^{T_1}(t)|^2dt+\int_u^{T_1}\int_{T_1}^T|Z(t,s)|^2dsdt\right]  \nonumber \\
&\leq &CE\int_u^{T_1}|\Psi (t)|^2dt+CE\int_u^{T_1}\left(
\int_{T_1}^T|g_0(t,s)|ds\right) ^2dt  \nonumber \\
&&+CE\int_u^{T_1}\left( \int_{T_1}^TL(t,s)|Y(s)|ds\right) ^2dt
\nonumber \\
 &\leq& CE\int_u^T|\Psi (t)|^2dt+CE\int_u^T\left(
\int_t^T|g_0(t,s)|ds\right) ^2dt,
\end{eqnarray}
then from (22) and (24), $\forall u\in [T_2,T_1],$
\begin{eqnarray}
&&\ E\left\{
\int_u^{T_1}|Y(t)|^2dt+\int_u^{T_1}\int_u^{T_1}|Z(t,s)|^2dsdt\right\}
\nonumber
 \\
\ &\leq &CE\int_u^T|\Psi (t)|^2dt+CE\int_u^T\left(
\int_t^T|g_0(t,s)|ds\right) ^2dt.
\end{eqnarray}
Here $C$ depends on $\sup\limits_{t\in
[0,T]}\int_t^TL(t,s)^{2+\epsilon }ds$. Since we are considering the
symmetrical form of $Z(\cdot ,\cdot )$, by the stochastic Fubini
Theorem we get:
\begin{eqnarray}
&&E\int_u^{T_1}\int_{T_1}^T|Z(t,s)|^2dsdt  \nonumber \\
&=&E\int_{T_1}^T\int_u^{T_1}|Z(t,s)|^2dtds  \nonumber \\
\ &=&E\int_{T_1}^T\int_u^{T_1}|Z(s,t)|^2dtds  \nonumber \\
\ &=&E\int_{T_1}^T\int_u^{T_1}|Z(t,s)|^2dsdt.
\end{eqnarray}
From (21), (24), (25) and (26), we can estimate that, $\forall u\in
[T_2,T_1],$
\begin{eqnarray*}
&&\ E\left\{ \int_u^T|Y(t)|^2dt+\int_u^T\int_u^T|Z(t,s)|^2dsdt\right\} \\
\ &=&E\left\{
\int_{T_1}^T|Y(t)|^2dt+\int_{T_1}^T\int_{T_1}^T|Z(t,s)|^2dsdt\right\} \\
&&+E\left\{
\int_u^{T_1}|Y(t)|^2dt+\int_u^{T_1}\int_u^{T_1}|Z(t,s)|^2dsdt\right\} \\
&&+E\int_{T_1}^T\int_u^{T_1}|Z(t,s)|^2dsdt+E\int_u^{T_1}%
\int_{T_1}^T|Z(t,s)|^2dsdt \\
\ &\leq &C\left\{ E\int_u^T|\Psi (t)|^2dt+CE\int_u^T\left(
\int_t^T|g_0(t,s)|ds\right) ^2dt\right\} .
\end{eqnarray*}
Thus we can repeat the argument above to obtain the estimate.
\end{proof}

\subsection{Some corollaries for S-solutions}
In this subsection we give some corollaries. Similar to \cite{Y2},
we easily claim the following results. We omit their proof.

\begin{corollary}
Let $\overline{g}:\Delta ^c\times R^m\times R^{m\times d}\times
R^{m\times
d}\times \Omega \rightarrow R^m$ also satisfies (H2). Let $\overline{\Psi }%
(\cdot )\in L_{\mathcal{F}_T}^2[0,T]$ and $(\overline{Y}(\cdot ),\overline{Z}%
(\cdot ,\cdot ))\in ^{*}\mathcal{H}^2[0,T]$ be the adapted
S-solution of
(1) with $g$ and $\Psi (\cdot )$ replaced by $\overline{g}$ and $\overline{%
\Psi }(\cdot )$, respectively, then we have the following: $\forall
S\in [0,T],$
\begin{equation}
\ E\int_S^T|Y(t)-\overline{Y}(t)|^2dt+E\int_S^T\int_S^T|Z(t,s)-\overline{Z}%
(t,s)|^2dsdt  \nonumber
\end{equation}
\[
\ \leq CE\int_S^T|\Psi (t)-\overline{\Psi }(t)|^2dt+
CE\int_S^T\left( \int_t^T|g-\overline{g}%
|ds\right) ^2dt,
\]
where $g=g(t,s,Y(s),Z(t,s),Z(s,t))$ and
$\overline{g}=\overline{g}(t,s,Y(s),Z(t,s),Z(s,t)).$
\end{corollary}
\begin{corollary}
Let (H2) hold and let $(Y(\cdot ),Z(\cdot ,\cdot ))\in
^{*}\mathcal{H}^2[0,T]$ be the adapted S-solution of (1). For any
$t\in [0,T]$, let $(\lambda ^t(\cdot ),\mu ^t(\cdot ))\in
\mathbb{H}^2[t,T]$ be the adapted solution of the following BSDE:
\begin{equation*}
\lambda ^t(r)=\Psi (t)+\int_r^Tg(t,s,Y(s),\mu
^t(s),Z(s,t))ds-\int_r^T\mu ^t(s)dW(s),\quad r\in [t,T]
\end{equation*}
Let
\begin{equation*}
\left\{
\begin{array}{lc}
\overline{Y}(t)=\lambda ^t(t), & t\in [0,T], \\
\overline{Z}(t,s)=\mu ^t(s), & (t,s)\in \Delta ^c.%
\end{array}
\right.
\end{equation*}
and let the values $\overline{Z}(t,s)$ of $\overline{Z}(\cdot ,\cdot )$ for $%
(t,s)\in \Delta $ be defined through:
\begin{equation*}
\overline{Z}(t,s)=\overline{Z}(s,t),\quad (t,s)\in \Delta
\end{equation*}
Then
\begin{equation*}
\left\{
\begin{array}{lc}
\overline{Y}(t)=Y(t), & t\in [0,T], \\
\overline{Z}(t,s)=Z(t,s), & (t,s)\in [0,T]^2.%
\end{array}
\right.
\end{equation*}
\end{corollary}
\begin{corollary}
Let (H2) hold, and $\Psi (\cdot )\in L_{\mathcal{F}_T}^2[0,T]$ and $%
(Y(\cdot ),Z(\cdot ,\cdot ))\in ^{*}\mathcal{H}^2[0,T]$ be the
unique adapted S-solution of BSVIE (1) on $[0,T]$, then for all
$S\in [0,T)$,
\begin{equation*}
\Psi ^S(t)=\Psi
(t)+\int_S^Tg(t,s,Y(s),Z(t,s),Z(s,t))ds-\int_S^TZ(t,s)dW(s)
\end{equation*}
is $\mathcal{F}_S$-measurable for almost all $t\in [0,S]$.
\end{corollary}
Next we will give an estimate to the S-solution of (1) which is
stronger than (13). We assume (1) admits a unique S-solution
$(Y(\cdot),Z(\cdot,\cdot)\in^{*}\mathcal {H}^2[0,T]$, and we will
estimate: $E\left( |Y(t)|^2+\int_t^T|Z(t,s)|^2ds\right).$ The author
gave such an estimate of M-solution in \cite{Y2}, but there is a
term $E\int_t^T|Z(s,t)|^2ds$ that can not be estimated directly so
he introduced the Malliavin calculus to treat this problem. But here
we don't have this kind of problem. In the following we assume
$(Y(\cdot),Z(\cdot,\cdot)$ is S-solution of (1). From (1) we have:
\begin{equation*}
Y(t)+\int_t^TZ(t,s)dW(s)=\Psi
(t)+\int_t^Tg(t,s,Y(s),Z(t,s),Z(s,t))ds,
\end{equation*}
thus
\begin{eqnarray*}
&&E\left\{ |Y(t)|^2+\int_t^T|Z(t,s)|^2ds\right\}   \nonumber
\\
&\leq& 2E|\Psi (t)|^2+4E\left( \int_t^T|g_0(t,s)|ds\right) ^2
\\
&&+4E\left( \int_t^T|g(t,s,Y(s),Z(t,s),Z(s,t))-g_0(t,s)|ds\right) ^2
\\
&\leq& 2E|\Psi (t)|^2+4E\left( \int_t^T|g_0(t,s)|ds\right) ^2
\\
&&+4(T-t)^{\frac \epsilon {2+\epsilon }}\sup\limits_{t\in
[0,T]}\left( \int_t^TL(t,s)^{2+\epsilon }ds\right) ^{\frac
2{2+\epsilon }}E\int_t^T|Y(s)|^2ds
\\
&&+8(T-t)^{\frac \epsilon {2+\epsilon }}\sup\limits_{t\in
[0,T]}\left( \int_t^TL(t,s)^{2+\epsilon }ds\right) ^{\frac
2{2+\epsilon }}E\int_t^T|Z(t,s)|^2ds.
\end{eqnarray*}
So there exists a constant $\eta =T-T_1,$ such that $\forall t\in
[T_1,T],$
\begin{eqnarray*}
&&4(T-t)^{\frac \epsilon {2+\epsilon }}\sup\limits_{t\in
[0,T]}\left(
\int_t^TL(t,s)^{2+\epsilon }ds\right) ^{\frac 2{2+\epsilon }} \\
&\leq &4\eta^{\frac \epsilon {2+\epsilon }}\sup\limits_{t\in
[0,T]}\left( \int_t^TL(t,s)^{2+\epsilon }ds\right) ^{\frac
2{2+\epsilon }}=\frac 13.
\end{eqnarray*}
Then we have, $\forall t\in [T_1,T],$
\begin{eqnarray}
&&E|Y(t)|^2 \leq 2E|\Psi (t)|^2+4E\left( \int_t^T|g_0(t,s)|ds\right)
^2+\frac 13E\int_t^T|Y(s)|^2ds
\end{eqnarray}
and
\begin{eqnarray}
&&E\int_t^T|Z(t,s)|^2ds  \nonumber \\
&\leq& 6E|\Psi (t)|^2+12E\left( \int_t^T|g_0(t,s)|ds\right)
^2+E\int_t^T|Y(s)|^2ds,\nonumber \\
\end{eqnarray}
thus from (27) and (28), we have:
\begin{eqnarray}
&&E|Y(t)|^2+E\int_t^T|Z(t,s)|^2ds \nonumber \\
&\leq &8E|\Psi (t)|^2+16E\left( \int_t^T|g_0(t,s)|ds\right) ^2+\frac
43E\int_t^T|Y(s)|^2ds  \notag \\
&\leq &8E|\Psi (t)|^2+\frac 43CE\int_t^T|\Psi (s)|^2ds  \notag \\
&&+16E\left( \int_t^T|g_0(t,s)|ds\right) ^2+\frac 43CE\int_t^T\left(
\int_s^T|g_0(s,u)|du\right) ^2ds.
\end{eqnarray}
The second inequality in (29) holds because of (13). $T$ is a finite
constant, from the method of choice $T_1$, $T_1$ depends only on
$\sup\limits_{t\in [0,T]}\left( \int_t^TL(t,s)^{2+\epsilon
}ds\right) ^{\frac 2{2+\epsilon }}$, thus there must exist finite
partition on $[0,T]$, $T=T_0\geq
T_1\geq T_2\geq \cdots \geq T_k=0$, such that $T_i-T_{i+1}\leq \eta ,$ and $%
\forall t\in [T_{i+1},T_i]$ $(i=0,1,2,\cdots k-1),$ we have the
following:
\begin{eqnarray}
&&E|Y(t)|^2+E\int_t^{T_i}|Z(t,s)|^2ds \nonumber \\
 &\leq& 8E|\Psi
^{T_i}(t)|^2+16E\left( \int_t^{T_i}|g_0(t,s)|ds\right) ^2+\frac
43E\int_t^{T_i}|Y(s)|^2ds
\nonumber \\
 &\leq& 8E|\Psi ^{T_i}(t)|^2+16E\left(
\int_t^{T_i}|g_0(t,s)|ds\right) ^2+\frac 43E\int_t^T|Y(s)|^2ds
\nonumber
\end{eqnarray}
\begin{eqnarray}
&\leq& 8E|\Psi ^{T_i}(t)|^2+\frac 43CE\int_t^T|\Psi (s)|^2ds
\nonumber
\\
&&+16E\left( \int_t^{T_i}|g_0(t,s)|ds\right) ^2+\frac
43CE\int_t^T\left( \int_s^T|g_0(s,u)|du\right) ^2ds.
\end{eqnarray}
Here \begin{equation*} Y(t)=\Psi
^{T_i}(t)+\int_t^{T_i}g(t,s,Y(s),Z(t,s),Z(s,t))ds-\int_t^{T_i}Z(t,s)dW(s),
\end{equation*}
where
\begin{equation}
\Psi ^{T_i}(t)=\Psi
(t)+\int_{T_i}^Tg(t,s,Y(s),Z(t,s),Z(s,t))ds-\int_{T_i}^TZ(t,s)dW(s).
\end{equation}
Now we will give estimates for $E|\Psi ^{T_i}(t)|^2$ and $%
E\int_{T_i}^T|Z(t,s)|^2ds, \quad t\in [T_{i+1},T_i],\quad
i=1,2,\cdots k-1.$ From Lemma 2.6 we have, $\forall t\in
[T_{i+1},T_i],$
\begin{eqnarray}
&&E|\Psi ^{T_i}(t)|^2+E\int_{T_i}^T|Z(t,s)|^2ds \nonumber \\
&\leq &CE|\Psi (t)|^2+CE\left(
\int_{T_i}^T|g(t,s,Y(s),0,0)|ds\right) ^2
\notag \\
&\leq &CE|\Psi (t)|^2+CE\left( \int_{T_i}^T|g(t,s,0,0,0)|ds\right)
^2  \notag\\
&&+CE\left( \int_{T_i}^TL(t,s)|Y(s)|ds\right) ^2  \notag \\
&\leq &CE|\Psi (t)|^2+CE\left( \int_t^T|g_0(t,s)|ds\right)
^2+CE\int_t^T|\Psi (s)|^2ds  \notag \\
&&+C\int_t^T\left( \int_s^T|g_0(s,u)|du\right) ^2ds.
\end{eqnarray}
Hence we have for any $t\in [0,T],$ there must exist one $i$, such
that $t\in [T_{i+1},T_i]$. From (30) and (32), we have,
\begin{eqnarray}
&&E|Y(t)|^2+E\int_t^T|Z(t,s)|^2ds \notag \\
&=&E|Y(t)|^2+E\int_t^{T_i}|Z(t,s)|^2ds+E\int_{T_i}^T|Z(t,s)|^2ds
\notag\\
 &\leq &l_1E\left( \int_t^T|g_0(t,s)|ds\right) ^2+l_2E|\Psi
(t)|^2
\notag\\
 &&+l_3E\int_t^T|\Psi (s)|^2ds+l_4E\int_t^T\left(
\int_s^T|g_0(s,u)|du\right) ^2ds \notag,
\end{eqnarray}
where $l_1,$ $l_2,$ $l_3,$ $l_4$ depend on $T$ and
$\sup\limits_{t\in [0,T]}\int_t^TL(t,s)^{2+\epsilon }ds.$

To sum up the argument above, we give the estimate as follows:
\begin{corollary}
Let $(Y(\cdot ),Z(\cdot ,\cdot ))$ be the S-solution of (1). Assume
\begin{equation*}
E\int_0^T\left( \int_s^T|g_0(s,u)|du\right) ^2ds<\infty ,\quad
E\int_0^T|\Psi (s)|^2ds<\infty ,
\end{equation*}
and for any $t\in[0,T]$, $E|\Psi (t)|^2<\infty,$ then we have
$\forall t\in [0,T],$
\begin{eqnarray}
&&\ E|Y(t)|^2+E\int_t^T|Z(t,s)|^2ds \notag \\
&\leq &l_1E\left( \int_t^T|g_0(t,s)|ds\right) ^2+l_2E|\Psi (t)|^2  \notag \\
&&+l_3E\int_t^T|\Psi (s)|^2ds+l_4E\int_t^T\left(
\int_s^T|g_0(s,u)|du\right) ^2ds,
\end{eqnarray}
where $l_1,$ $l_2,$ $l_3,$ $l_4$ depend on $T$ and
$\sup\limits_{t\in [0,T]}\int_t^TL(t,s)^{2+\epsilon }ds.$
\end{corollary}
We also have:
\begin{corollary}
Let $(Y(\cdot ),Z(\cdot ,\cdot ))$ be the S-solution of (1). Assume
\begin{equation*}
\sup\limits_{t\in [0,T]}E|\Psi (t)|^2<\infty,\sup\limits_{t\in
[0,T]}E\left( \int_t^T|g_0(t,s)|ds\right) ^2<\infty,
\end{equation*}
then $\sup\limits_{t\in [0,T]}E|Y(t)|^2<\infty $ and
$\sup\limits_{t\in [0,T]}E\int_t^T|Z(t,s)|^2ds<\infty.$
\end{corollary}
Next we show the continuity of $Y(t)$ in $t$. We have:
\begin{corollary}
Let $(Y(\cdot ),Z(\cdot ,\cdot ))$ be the S-solution of (1), and
assume
\begin{equation}
\sup\limits_{t\in [0,T]}E\left( \int_t^T|g_0(t,s)|ds\right)
^2<\infty , \quad \sup\limits_{t\in [0,T]}E|\Psi (t)|^2<\infty ,
\end{equation}
then $\forall t,\overline{t}\in [0,T],$ we have
\begin{eqnarray}
&&E\left\{ |Y(t)-Y(\overline{t})|^2+\int_{t\vee \overline{t}}^T|Z(t,s)-Z(%
\overline{t},s)|^2ds\right\}  \notag \\
&\leq &CE\left\{ |\Psi (t)-\Psi (\overline{t})|^2+\int_{t\wedge \overline{t}%
}^{t\vee \overline{t}}|Z(t\wedge \overline{t},s)|^2ds\right.  \notag \\
&&+\left( \int_{t\vee \overline{t}}^T|g(t,s,Y(s),Z(t,s),Z(s,t))-g(\overline{t%
},s,Y(s),Z(t,s),Z(s,t))|ds\right) ^2  \notag \\
&&+\left. \left( \int_{t\wedge \overline{t}}^{t\vee
\overline{t}}|g(t\wedge
\overline{t},s,Y(s),Z(t\wedge \overline{t},s),Z(s,t\wedge \overline{t}%
))|ds\right) ^2\right\} .
\end{eqnarray}
Consequently, in the case that
\begin{equation}
\lim\limits_{|t-\overline{t}|\rightarrow 0}E|\Psi (t)-\Psi (\overline{t}%
)|^2=0,
\end{equation}
and $t\mapsto g(t,s,y,z,\zeta )$ is continuous in the sense that
\begin{eqnarray}
|g(t,s,y,z,\zeta )-g(\overline{t},s,y,z,\zeta )| &\leq
&C(1+|y|+|z|+|\zeta
|)\rho (|t-\overline{t}|), \\
\text{ } \forall t,\overline{t} &\in &[0,T],\text{ }s\in [t\vee
\overline{t},T],\text{ }y,z,\zeta \in R,  \notag
\end{eqnarray}
for some modulus of continuity $\rho (\cdot ),$ then we have
\begin{equation}
\lim\limits_{|t-\overline{t}|\rightarrow 0}\left(E|Y(t)-Y(\overline{t}%
)|^2+E\int_{t\vee
\overline{t}}^T|Z(t,s)-Z(\overline{t},s)|^2ds\right )=0.
\end{equation}
\end{corollary}
\begin{proof}We can easily obtain (35) by (9) with
$h(t,s,z)\equiv g(t,s,Y(s),z,z)$. We have
\begin{eqnarray}
&&\ E\left( \int_{t\wedge \overline{t}}^{t\vee
\overline{t}}|g(t\wedge
\overline{t},s,Y(s),Z(t\wedge \overline{t},s),Z(s,t\wedge \overline{t}%
))|ds\right) ^2 \nonumber \\
\ &\leq &CE\left( \int_{t\wedge \overline{t}}^{t\vee \overline{t}%
}|g_0(t\wedge \overline{t},s)|ds\right) ^2+ C\int_{t\wedge \overline{t}%
}^{t\vee \overline{t}}|L(t\wedge \overline{t},s)|^2ds\cdot
E\int_{t\wedge
\overline{t}}^{t\vee \overline{t}}|Y(s)|^2ds  \nonumber \\
&&+ C\int_{t\wedge \overline{t}}^{t\vee \overline{t}}|L(t\wedge \overline{t}%
,s)|^2ds\cdot E\int_{t\wedge \overline{t}}^{t\vee
\overline{t}}|Z(t\wedge
\overline{t},s)|^2ds  \nonumber \\
&&+ C\int_{t\wedge \overline{t}}^{t\vee \overline{t}}|L(t\wedge \overline{t}%
,s)|^2ds\cdot E\int_{t\wedge \overline{t}}^{t\vee
\overline{t}}|Z(s,t\wedge \overline{t})|^2ds.
\end{eqnarray}
From Corollary 5 we have
\begin{equation}
\sup\limits_{t\in [0,T]}E\int_t^T|Z(t,s)|^2ds<\infty .
\end{equation}
From (36), (37), (39) and (40), (38) is obtained.
\end{proof}

\subsection{The relations between S-solutions and other solutions}
Let us consider the following BSVIE which is a generalization of
BSVIE in \cite{L}.
\begin{equation}
Y(t)=\Psi (t)+\int_t^Tf(t,s,Y(s),Z(t,s))ds-\int_t^TZ(t,s)dW(s).
\end{equation}
First we give a definition of the adapted solutions of BSVIEs.
\begin{definition}
Let $S\in [0,T]$. A pair of $(Y(\cdot ),Z(\cdot ,\cdot ))\in \mathcal{H}%
^2_1[S,T]$ is called an adapted solution of BSVIE (41) on $[S,T]$ if
(41) holds in the usual It\^o's sense for almost all $t\in [S,T].$
\end{definition}
There is a gap in \cite{L}. Now we can easily prove the existence
and uniqueness of adapted solution of (41) which is a generalization
of the result in \cite{L} and overcome the gap in \cite{L}. We can
claim:
\begin{theorem}
Let $f:\Delta ^c\times R^m\times R^{m\times d}\times \Omega
\rightarrow R^m$ be $\mathcal{B}(\Delta ^c\times R^m\times
R^{m\times d})\otimes \mathcal{F}_T $-measurable such that $s\mapsto
f(t,s,y,z)$ is $\mathbb{F}$-progressively measurable for all
$(t,y,z)\in [0,T]\times R^m\times R^{m\times d}$ and
\begin{equation*}
E\int_0^T\left( \int_t^T|f_0(t,s)|ds\right) ^2dt<\infty,
\end{equation*}
where $f_0(t,s)\equiv f(t,s,0,0).$ Moreover, we assume $\forall
(t,s)\in \Delta ^c$, $y,$ $\overline{y}\in R^m$, $z,$
$\overline{z}\in R^{m\times d},$
\begin{equation*}
|f(t,s,y,z)-f(t,s,\overline{y},\overline{z})|\leq L(t,s)(|y-\overline{y}|+|z-%
\overline{z}|),
\end{equation*}
where $L:\Delta ^c\rightarrow R$ is a deterministic function so that
for some $\epsilon >0,$
\begin{equation*}
\sup\limits_{t\in [0,T]}\int_t^TL(t,s)^{2+\epsilon }ds<\infty .
\end{equation*}
If $\Psi (\cdot )\in L_{\mathcal{F}_T}^2[0,T]$, then (41) admits a
unique adapted solution.
\end{theorem}
\begin{proof}
Let $ g(t,s,Y(s),Z(t,s),Z(s,t))\equiv
f(t,s,Y(s),\frac{Z(t,s)}2+\frac{Z(s,t)}2)$. We consider the
following BSVIE, $t\in [0,T],$
\begin{equation}
Y(t)=\Psi
(t)+\int_t^Tg(t,s,Y(s),Z(t,s),Z(s,t))ds-\int_t^TZ(t,s)dW(s).
\end{equation}
$\forall y,$ $\overline{y},$ $z,$ $\overline{z},$ $\zeta ,$
$\overline{\zeta }\in R,$ we have
\begin{eqnarray*}
&&\left| g(t,s,y,z,\zeta )-g(t,s,\overline{y},\overline{z},\overline{\zeta }%
)\right|  \\
&=&\left| f(t,s,y,,\frac z2+\frac \zeta 2)-f(t,s,\overline{y},\frac{%
\overline{z}}2+\frac{\overline{\zeta }}2)\right|  \\
&\leq &L(t,s)\left( \left| y-\overline{y}\right| +\left| \frac
z2+\frac
\zeta 2-\frac{\overline{z}}2-\frac{\overline{\zeta }}2\right| \right)  \\
&\leq &L(t,s)(|y-\overline{y}|+|z-\overline{z}|+|\zeta
-\overline{\zeta }|).
\end{eqnarray*}
So $g$ satisfies (H2). Then BSVIE (42) admits a unique S-solution on
$[0,T]$ by Theorem 3.1. Then we obtain the existence and uniqueness
of the adapted solution of (41) on $[0,T]$.
\end{proof}
\begin{remark}
On the other hand, for the S-solution of (1), due to $Z(t,s)\equiv
Z(s,t)$, we can let
$$f(t,s,Y(s),Z(t,s))\equiv g(t,s,Y(s),Z(t,s),Z(s,t)),$$ then (1) can be
transformed into (41). If we have obtained the adapted solution
$(Y(\cdot),Z(\cdot,\cdot))\in\mathcal {H}^2_1[0,T]$ for (41), then
we can get the S-solution of (1) in $^{*}\mathcal {H}^2[0,T]$ by
defining the value for $Z(t,s)\equiv Z(s,t).\quad (t,s)\in \Delta.$
\end{remark}
\begin{remark}
If BSVIE (1) degenerates to BSVIE (41), then the M-solution $(Y_1,
Z_1)$ and S-solution $(Y_2, Z_2)$ of (1) are identical in
$\Delta^c$, i.e., $Y_1(t)\equiv Y_2(t)$, $Z_1(t,s)\equiv Z_2(t,s),
\quad 0\leq t\leq s\leq T.$ But $Z_1(t,s)$ and $Z_2(t,s)$ may be
different in $\Delta.$ Now we give two examples to illustrate it.
Let's consider the following BSVIE
\begin{equation}
Y(t)=tTW(T)-\int_t^TtY(s)/s^2ds-\int_t^TZ(t,s)dW(s),\quad t\in
[T_1,T].
\end{equation}
Here $T_1>0$, $\Psi (t)=tTW(T)$,
$g(t,s,Y(t),Z(t,s),Z(s,t))=-tY(s)/s^2.$ It
is easy to check that $E\int_{T_1}^Tt^2T^2W^2(T)dt=(T^6-T^3T_1^3)/3<\infty $%
, and $g$ satisfies the assumption (H2).

It is obvious that $Y(t)=t^2W(t),$ $Z(t,s)=ts$ satisfies (43), thus
it is the unique S-solution of (43). We also know that BSVIE (43)
has a unique M-solution (see \cite{Y2}). But the M-solution is not
equal to the S-solution of (43). In fact, if the unique S-solution
of (43) also is the M-solution of (43), we have
\begin{eqnarray*}
t^2W(t)=Y(t)&=&E(\left. %
Y(t)\right| \mathcal{F}_{T_{1}})+\int_{T_1}^tZ(t,s)dW(s) \\
&=&t^2W(T_1)+\int_{T_1}^ttsdW(s)\\
&=&t^2W(T_1)+t^2W(t)-tW(T_1)T_1-t\int_{T_1}^tW(s)ds.
\end{eqnarray*}
 Thus
$$tW(T_1)T_1+t\int_{T_1}^tW(s)ds-t^2W(T_1)=0,\quad \forall t\in
[T_1,T],$$  then $$\frac{\int^{t_1}_{t_2}W(s)ds}{t_1-t_2}=W(T_1),
\quad \forall t_1, t_2\in [T_1,T],$$  which means that $W(t)=W(T_1)$
for any $t\in[T_1,T]$. Obviously it is a contradiction.

Now we give the explicit M-solution for (43). Let $Z(t,s)=ts,$
$T_1\leq t\leq s\leq T$, and $Y(t)=t^2W(t),$ $t\in[T_1,T].$ Because
$$E\int^t_{T_{1}}|D_sY(t)|^2ds=\int^t_{T_{1}}t^4I_{[0,t]}(s)ds<\infty,$$
then by Ocone-Clark formula (see \cite{D}) and the definition of
M-solution, we have $$Y(t)=E(\left.Y(t)\right|
\mathcal{F}_{T_{1}})+\int_{T_1}^tE(\left.D_sY(t)\right|
\mathcal{F}_{s})ds=E(\left.Y(t)\right|
\mathcal{F}_{T_{1}})+\int_{T_1}^tZ(t,s)dW(s).$$ Thus
\begin{equation*}
Z(t,s)=E(D_st^2W(t)|\mathcal{F}_s)=t^2,\text{ }T_1\leq s<t\leq T.
\end{equation*}
Therefore we obtain the M-solution of (43) as follows:

\begin{equation*}
\left\{
\begin{array}{lc}
Y(t)=t^2W(t), & t\in [T_1,T], \\
Z(t,s)=ts, & t,s\in\Delta^c[T_1,T], \\
Z(t,s)=t^2, & t,s\in\Delta[T_1,T].%
\end{array}
\right.
\end{equation*}

The above example is on $[T_1,T],\quad (T_1>0)$. Now we give an
example on $[0,T].$ Let's consider the following BSVIE, $t\in
[0,T],$
\begin{equation}
Y(t)=W(T)(T+1)(t+1)-\int_t^T\frac{(t+1)Y(s)}{(s+1)^2}ds-\int_t^TZ(t,s)dW(s),
\end{equation}
By the same method we can get the unique S-solution
$Y(t)=(t+1)^2W(t); Z(t,s)=(t+1)(s+1)$ of (44). From
\begin{equation*}
Z(t,s)=E(D_s(t+1)^2W(t)|\mathcal{F}_s)=(t+1)^2,\text{ }0\leq s<t\leq
T.
\end{equation*}
we know the unique M-solution of (44) is as follows:
\begin{equation*}
\left\{
\begin{array}{lc}
Y(t)=(t+1)^2W(t), & t\in [0,T], \\
Z(t,s)=(t+1)(s+1), & t,s\in\Delta^c, \\
Z(t,s)=(t+1)^2, & t,s\in\Delta.%
\end{array}
\right.
\end{equation*}
\end{remark}

\begin{remark}
From Remark 2 we know that when the generator of BSVIE is
independent of $Z(s,t)$ $0\leq s<t\leq T$, the M-solution and
S-solution can be equal in $\Delta^c$. However, the following
example show that when the generator depends on $Z(s,t)$ $0\leq
s<t\leq T$, the M-solution and S-solution can also be equal.
Let us consider the following BSVIE:
\begin{equation*}
Y(t)=\int_t^Tg(t,s,Y(s),Z(s,t))ds-\int_t^TZ(t,s)dW(s),\text{ }t\in
[0,T].
\end{equation*}
Here we assume (H2) holds and $g(t,s,0,0)\equiv 0$. We can easily
check that $Y(t)\equiv 0,Z(t,s)\equiv 0,$ $t,s\in [0,T]$ is not only
the unique M-solution, but also the unique S-solution.
\end{remark}

\subsection{An interesting result for S-solutions}

Now we give an interesting result for S-solutions. We consider the
following BSVIE: $t\in [0,T]$
\begin{equation}
Y(t)=\Psi
(t)+\int_t^Tg(t,s,Y(s),Z(t,s),Z(s,t))ds-\int_t^TZ(s,t)dW(s).
\end{equation}
We denote
\begin{equation*}
\overline{\mathcal{H}}^2[R,S]=L_{\mathbb{F}}^2[R,S]\times \overline{L}^2([R,S];L_{\mathbb{F}%
}^2[R,S]).
\end{equation*}
Here $\overline{L}^2([R,S];L_{\mathbb{F}}^2[R,S])$ is the set of all processes $%
z:[R,S]^2\times \Omega \rightarrow R^{m\times d}$ such that for almost all $%
t\in [R,S]$, $z(\cdot,t)\in L_{\mathbb{F}}^2[R,S]$ satisfying
\begin{equation*}
E\int_R^S\int_R^S|z(s,t)|^2dsdt<\infty.
\end{equation*}

We can define the norm of $\overline{\mathcal{H}}^2[R,S]$ as the
norm of $\mathcal{H}^2[R,S]$.
 We can define S-solution for (45). Obviously
(45) has a unique S-solution which is the same as the one of (1). By
the same method as in \cite{Y2}, we can also prove (45) admits a
unique M'-solution defined as follows.
\begin{definition}
Let $S\in [0,T]$. A pair of $(Y(\cdot ),Z(\cdot ,\cdot ))\in \overline{\mathcal{H}}%
^2_2[S,T]$ is called an adapted M'-solution of (45) on $[S,T]$, if
(45) holds in the usual It$\hat o$'s sense for almost all $t\in
[S,T]$ and, in addition, the following holds:
\begin{equation*}
Y(t)=E[Y(t)|\mathcal{F}_S]+\int_S^tZ(s,t)dW(s),\text{ a.e. }t\in
[S,T].
\end{equation*}
\end{definition}
We have the following proposition.
\begin{proposition}
We consider the following two equations: $t\in [0,T]$
\begin{equation}
Y_1(t)=\Psi
_1(t)+\int_t^Tg_1(t,s,Y_1(s),Z_1(t,s),Z_1(s,t))ds-\int_t^TZ_1(t,s)dW(s).
\end{equation}
\begin{equation}
Y_2(t)=\Psi
_2(t)+\int_t^Tg_2(t,s,Y_2(s),Z_2(t,s),Z_2(s,t))ds-\int_t^TZ_2(s,t)dW(s).
\end{equation}
We assume that for any $0\leq t\leq s\leq T$ and $y\in R^m, z,
\zeta\in R^{m\times d},$ $$\Psi _1(t)\equiv \Psi _2(t),\text{ }
g_1(t,s,y,z,\zeta)\equiv g_2(t,s,y,z,\zeta), \text{ } a.s.$$  then
we have
\begin{equation*}
Y_1(t)\equiv Y_2(t),\text{ }Z_1(t,s)\equiv Z_2(t,s),\text{ } \forall
t,s\in [0,T],
\end{equation*}
if and only if $(Y_i(\cdot ),Z_i(\cdot ,\cdot ))$, $(i=1,2)$, are
the S-solutions of (46) and (47), respectively.
\end{proposition}

\begin{proof}$\Leftarrow $: It is clear.

$\Rightarrow $: If $Y_1(t)\equiv Y_2(t)$, $Z_1(t,s)\equiv Z_2(t,s)$,
$\forall t,s\in[0,T],$ thus
$$g_1(t,s,Y_1(s),Z_1(t,s),Z_1(s,t))=g_2(t,s,Y_2(s),Z_2(t,s),Z_2(s,t)),$$
then $\int_t^T\left( Z_1(t,s)-Z_2(s,t)\right) dW(s)\equiv 0$,
furthermore,
\[
E\left( \int_t^T\left( Z_1(t,s)-Z_2(s,t)\right) dW(s)\right)
^2=E\int_t^T(Z_1(t,s)-Z_2(s,t))^2ds\equiv 0,\ \ t\in [0,T].
\]
So we have $Z_1(t,s)\equiv Z_2(s,t),$ $t\leq s.$ Because of the assumption of $%
Z_1(t,s)\equiv Z_2(t,s),$ $t,$ $s\in [0,T],$ we have $Z_2(t,s)\equiv Z_2(s,t),$ $(t\leq s)$%
, and the solution $(Y_2,Z_2)$ is the S-solution of (47). By a
similar method we can show $(Y_1,Z_1)$ is the S-solution of (46).
\end{proof}
\begin{remark}
From the proposition above, we know that when two kinds of equations
such as (46) and (47) have the same terminal condition and the same
generator, they have the same solution if and only if both of the
solutions are S-solution. Next we give an example to show this. We
consider the following two BSVIEs
\begin{eqnarray}
Y_1(t)
&=&tTW(T)-\int_t^T\frac{tY_1(s)}{s^2}ds-\int_t^TZ_1(t,s)dW(s),\quad
t\in [T_1,T]. \\
Y_2(t)
&=&tTW(T)-\int_t^T\frac{tY_2(s)}{s^2}ds-\int_t^TZ_2(s,t)dW(s),\quad
t\in [T_1,T].
\end{eqnarray}
Obviously (48) and (49) have the same S-solution, for $i=1,2$
\begin{equation*}
Y_i(t)=t^2W(t);\text{ }Z_i(t,s)=Z_i(s,t)=ts,\text{ }t,s\in[T_1,T],
\end{equation*}
however, the M-solution of (48) is not equal to the M'-solution of
(49). In fact, from Remark 2 we know that
$Y_1(t)=Y_2(t)=t^2W(t)\text{ } (t\in[T_1,T])$, $Z_1(t,s)=ts\text{ }
(t\leq s)$ and $Z_2(s,t)=ts\text{ } (t\leq s)$. We can also
determine $Z_1(t,s)\text{ } (t>s)$ by
\begin{equation}
Z_1(t,s)=E(D_st^2W(t)|\mathcal{F}_s)=t^2,\text{ }T_1\leq s<t\leq T.
\end{equation}
and $Z_2(s,t)\text{ } (t>s)$ by
\begin{equation}
Z_2(s,t)=E(D_st^2W(t)|\mathcal{F}_s)=t^2,\text{ }T_1\leq s<t\leq T.
\end{equation}
So we have $Z_1(t,s)=ts\neq s^2=Z_2(t,s)$ for $t<s$ and
$Z_1(t,s)=t^2\neq ts=Z_2(t,s)$ for $t>s$.

When neither the terminal conditions nor the generators are equal,
the conclusion above can hold too. For example, we can choose $\Psi
_1(t)=\Psi(t)+c$, $\Psi _2(t)=\Psi(t)-c$,
$g_1(t,s,y,z,\zeta)=g(t,s,y,z,\zeta)-c$ and
$g_2(t,s,y,z,\zeta)=g(t,s,y,z,\zeta)+c$, here $c>0$ is a constant.
We have $\Psi _1(t)\neq\Psi_2(t)$ and $g_1(t,s,y,z,\zeta)\neq
g_2(t,s,y,z,\zeta)$, but the conclusion still holds.
\end{remark}

\section{Dynamic risk measures by special BSVIEs}

In this section, we assume $m=d=1$ and $f$ is independent of $\omega
.$ We know that the following BSVIE admits a unique adapted
M-solution and a unique adapted S-solution when the generator and
the terminal condition satisfy certain conditions:
\begin{equation}
Y(t)=\psi (t)+\int_t^Tf(t,s,Y(s),Z(t,s))ds-\int_t^TZ(t,s)dW(s).
\end{equation}
From the definition of the M-solution and S-solution, we know that
both of them which solve (52) in the It\^o sense have the same value
in the following part
\begin{equation*}
(Y(t),Z(t,s)),\quad 0\leq t\leq s\leq T, \quad t\in [0,T],
\end{equation*}
and the only difference between the two kinds of solutions is the
value of
$$Z(t,s), \quad 0\leq s<t\leq T.$$
Now we will give a comparison theorem on S-solution for the
following BSVIE:
\begin{equation}
Y(t)=-\psi
(t)+\int_t^T(f(t,s,Y(s))+r_1(s)Z(t,s)+r_2(s)Z(s,t))ds-\int_t^TZ(t,s)dW(s).
\end{equation}
Here $r_i(s)$ are two deterministic functions which satisfy
$e^{\frac
12\int_0^Tr_i^2(s)ds}<\infty $. Thus we can determine the value $Y(t),$ $%
t\in [0,T]$ of S-solution to (53) by solving adapted solution of
(54)
\begin{equation}
Y(t)=-\psi
(t)+\int_t^T(f(t,s,Y(s))+(r_1(s)+r_2(s))Z(t,s))ds-\int_t^TZ(t,s)dW(s).
\end{equation}
And we can use Girsanov theorem to rewrite (54)
\begin{equation}
Y(t)=-\psi
(t)+\int_t^Tf(t,s,Y(s))ds-\int_t^TZ(t,s)d\widetilde{W}(s),
\end{equation}
where $\widetilde{W}(t)=W(t)+\int_0^t(r_1(s)+r_2(s))ds$ is a
Brownian motion under new probability measure $\widetilde{P}$
defined by
\begin{equation*}
\frac{d\widetilde{P}}{dP}(\omega )=\exp \left\{
\int_0^Tr(s)dW(s)-\frac 12\int_0^Tr^2(s)ds\right\} ,\text{
}r(s)=r_1(s)+r_2(s).
\end{equation*}
Before proving the comparison theorem for S-solution, we need the
following proposition in \cite{Y3}.

\begin{proposition}
We consider the following BSVIE
\begin{equation}
Y(t)=-\psi (t)+\int_t^Tf(t,s,Y(s),Z(s,t))ds-\int_t^TZ(t,s)dW(s).
\end{equation}
Let $f,$ $\overline{f}:$ $\Delta ^c\times R\times R^d\mapsto R$
satisfy (H2) (here we assume $L(t,s)$ is a bounded function),
and let $\psi (\cdot ),$ $\overline{\psi }(\cdot )\in L_{\mathcal{F}%
_T}^2[0,T]$ such that
\begin{eqnarray}
f(t,s,y,z) &\geq &\overline{f}(t,s,y,z),\text{ }\forall (t,s,y,z)\in
\Delta
^c\times R\times R^d, \\
\psi (t) &\leq &\overline{\psi }(t),\text{  a.s. } \quad t\in
[0,T],\text{ a.e.} \notag
\end{eqnarray}
Let $(Y(\cdot ),Z(\cdot ,\cdot ))$ be the adapted M-solution of
BSVIE (56), and $(\overline{Y}(\cdot ),\overline{Z}(\cdot ,\cdot ))$
be the adapted
M-solution of BSVIE (56) with $f$ and $\psi (\cdot )$ replaced by $%
\overline{f}$ and $\overline{\psi }(\cdot ),$ respectively. Then the
following holds:
\begin{equation*}
Y(t)\geq \overline{Y}(t),\text{a.s. } t\in [0,T],\text{ a.e.}
\end{equation*}
\end{proposition}
We then have
\begin{lemma}
Let $f,\overline{f}:$ $\Delta ^c\times R\mapsto R$ satisfy (H2) (here we assume $L(t,s)$ to be bounded), and let $%
\psi (\cdot ),$ $\overline{\psi }(\cdot )\in
L_{\mathcal{F}_T}^2[0,T]$ such that
\begin{eqnarray}
f(t,s,y) &\geq &\overline{f}(t,s,y),\text{ }\forall (t,s,y)\in
\Delta
^c\times R, \\
\psi (t) &\leq &\overline{\psi }(t),\text{a.s. }\quad t\in [0,T].
\notag
\end{eqnarray}
Let $(Y(\cdot ),Z(\cdot ,\cdot ))$ be the adapted S-solution of
BSVIE (53), and $(\overline{Y}(\cdot ),\overline{Z}(\cdot ,\cdot ))$
be the adapted
S-solution of BSVIE (53) with $f$ and $\psi (\cdot )$ replaced by $%
\overline{f}$ and $\overline{\psi }(\cdot ),$ respectively. Then the
following holds:
\begin{equation*}
Y(t)\geq \overline{Y}(t),\text{ a.s.,}\text{ } \forall t\in [0,T],
\end{equation*}
\end{lemma}
\begin{proof}We assume $(Y_1(\cdot ),Z_1(\cdot ,\cdot ))$ is the
unique M-solution of BSVIE (55), and $(\overline{Y}_1(\cdot ),\overline{Z%
}_1(\cdot ,\cdot ))$ is the unique M-solution of (55) with $f$ and
$\psi (\cdot )$ replaced by $\overline{f}$ and $\overline{\psi
}(\cdot ),$
respectively. Clearly, we have: $t\in [0,T],$ $$\widetilde{P}\{\omega ;%
Y_1(t)=Y(t)\}=1; \quad  \widetilde{P}\{\omega ; \overline{Y}_1(t)=\overline{%
Y}(t)\}=1.$$ From Proposition 3, we know $$\widetilde{P}\{\omega
;Y_1(t)\geq \overline{Y}_1(t)\}=1,$$ then we have: $\forall t\in
[0,T]$, $\widetilde{P}\{\omega ;Y(t)\geq \overline{Y}(t)\}=1$, thus
$P\{\omega ;Y(t)\geq \overline{Y}(t)\}=1.$
\end{proof}
In what follows, we define
\begin{equation}
\rho (t;\psi (\cdot ))=Y(t),\quad \forall t\in [0,T],
\end{equation}
where $(Y(\cdot ),Z(\cdot ,\cdot ))$ is the unique adapted
S-solution of BSVIE (53).

\begin{lemma}
Let $f:\Delta ^c\times R\mapsto R$ satisfy (H2).

1) Suppose $f$ is sub-additive, i.e.,
\begin{equation*}
f(t,s,y_1+y_2)\leq f(t,s,y_1)+f(t,s,y_2),\text{ } (t,s)\in \Delta ^c,%
\text{ }y_1,y_2\in R,\text{a.e.,}
\end{equation*}
then $\psi (\cdot )\mapsto \rho (t;\psi (\cdot ))$ is sub-additive,
i.e.,
\begin{equation*}
\rho (t;\psi _1(\cdot )+\psi _2(\cdot ))\leq \rho (t;\psi _1(\cdot
))+\rho (t;\psi _2(\cdot )),\text{a.s.,} \quad  t\in [0,T].\text{
a.e.}
\end{equation*}
\end{lemma}
\begin{proof}We can get the conclusion by Lemma 4.1. \end{proof}

\begin{lemma}
1) If the generator of (53) is: $f(t,s,y)=\eta (s)y,$ with $\eta
(\cdot )$ being a deterministic integrable function, then $\psi
(\cdot )\mapsto \rho (t;\psi (\cdot ))$ is translation invariant,
i.e.,
\begin{equation*}
\rho (t;\psi (\cdot )+c)=\rho (t;\psi (\cdot ))-ce^{\int_t^T\eta (s)ds},\text{ a.s., }%
\text{ }t\in [0,T], \quad \forall c\in R.
\end{equation*}
In particular, if $\eta (\cdot )=0,$ then
\begin{equation*}
\rho (t;\psi (\cdot )+c)=\rho (t;\psi (\cdot ))-c,\text{ }t\in
[0,T],\text{ a.s., }\forall c\in R.
\end{equation*}

2) If $f:\Delta ^c\times R\mapsto R$ is positively homogeneous, i.e., $%
f(t,s,\lambda y)=\lambda f(t,s,y),$ $t,s\in \Delta ^c,$ a.s.,
$\forall \lambda \in R^{+},$ so is $\psi (\cdot )\mapsto \rho
(t;\psi (\cdot )).$
\end{lemma}
\begin{proof}The result is obvious.
\end{proof}
We then have
\begin{theorem}
Suppose $f(t,s,y)=\eta (s)y,$ with $\eta (\cdot )$ being a
deterministic bounded function, then $\rho (\cdot )$ defined by (59)
is a dynamic coherent risk measure.
\end{theorem}
\begin{proof}
It is not difficult to obtain the conclusion by the above lemmas.
\end{proof}

\section*{Acknowledgments}

The authors would like to thank the referee and the editor for their
helpful comments and suggestions. The authors are also grateful to
Prof. Jiongmin Yong for providing the paper [16].


\end{document}